\documentclass[10pt]{amsart}

\usepackage{lmodern}
\usepackage[T1]{fontenc}

\usepackage{amsmath, amsthm, amscd, amssymb, xcolor, fancyhdr}
\definecolor{dblue}{rgb}{0,0,.6}
\usepackage[backref=page, colorlinks=true, linkcolor=dblue, citecolor=dblue, filecolor=dblue, menucolor=dblue, urlcolor=dblue]{hyperref}
\usepackage{tikz-cd}

\renewcommand*{\backrefalt}[4]{
	\ifcase #1 (Not cited)
	\or        (Cited on page~#2)
	\else      (Cited on pages~#2)
	\fi}

\newcommand{\calo}{\mathcal{O}}

\newcommand{\bbA}{\mathbb{A}}
\newcommand{\bbZ}{\mathbb{Z}}
\newcommand{\bbQ}{\mathbb{Q}}
\newcommand{\bbR}{\mathbb{R}}
\newcommand{\bbC}{\mathbb{C}}

\newcommand{\bbP}{\mathbb{P}}

\newcommand{\DD}{\mathcal{D}}

\newcommand{\HH}{\mathcal{H}}
\newcommand{\LL}{\mathcal{L}}

\newcommand{\OO}{\mathcal{O}}

\newcommand{\VV}{\mathcal{V}}

\newcommand{\XX}{\mathcal{X}}
\newcommand{\YY}{\mathcal{Y}}
\newcommand{\ZZ}{\mathcal{Z}}

\newcommand{\st}{\enskip |\enskip}

\newcommand{\cnv}{\,\lrcorner\,\,}

\newcommand{\emrp}{\mathrm{End}}

\newcommand{\dd}{\partial}
\newcommand{\ddb}{\overline\partial}

\newcommand{\ii}{i}

\newcommand{\wdg}{\wedge}

\newcommand{\lrarr}{\longrightarrow}
\newcommand{\hrarr}{\hookrightarrow}

\newtheorem{defn}{Definition}[section]

\newtheorem{thm}[defn]{Theorem}
\newtheorem{lem}[defn]{Lemma}
\newtheorem{cor}[defn]{Corollary}

{\theoremstyle{remark}
  \newtheorem{rem}[defn]{Remark}
  
}

\makeatletter
\@addtoreset{equation}{section}
\makeatother

\newcommand{\version}{Version 1.3, \today}
\title{The Moser isotopy for holomorphic symplectic and C-symplectic structures}

\author{Andrey Soldatenkov}
\address{Steklov Mathematical Institute of Russian Academy of Sciences\\
8 Gubkina Street\\ Moscow 119991\\ Russia}
\email{soldatenkov@mi-ras.ru}
\thanks{The work of Andrey Soldatenkov was performed
at the Steklov International Mathematical Center and supported
by the Ministry of Science and Higher Education of the Russian Federation
(agreement no. 075-15-2019-1614)}

\author{Misha Verbitsky}
\address{Instituto Nacional de Matem\'atica Pura e Aplicada\\
Estrada Dona Castorina, 110\\
Jardim Bot\^anico, CEP 22460-320\\
Rio de Janeiro, RJ - Brasil}
\address{Laboratory of Algebraic Geometry\\
National Research University HSE\\
Department of Mathematics, 6 Usacheva Street\\ Moscow, Russia
}
\email{verbit@impa.br}
\thanks{Misha Verbitsky is
partially supported by the HSE University Basic Research Program,
by FAPERJ grant E-26/202.912/2018 
and CNPq - Process 313608/2017-2.}

\begin{document}

\begin{abstract}
A C-symplectic structure is a complex-valued 
2-form which is holomorphically symplectic for
an appropriate complex structure.
We prove an analogue of Moser's isotopy theorem for families of
C-symplectic structures and list several applications of
this result. We prove that the degenerate twistorial
deformation associated to a holomorphic Lagrangian
fibration is locally trivial over the base of this
fibration. This is used to extend several theorems about Lagrangian
fibrations, known for projective hyperk\"ahler manifolds,
to the non-projective case. We also exhibit new examples
of non-compact complex manifolds with infinitely many pairwise
non-birational algebraic compactifications. 
\end{abstract}

\maketitle
\fancypagestyle{plain}{
  \fancyhf{}
  \cfoot{\small -- \thepage \ --} \rfoot{\tiny \sc\version}
  \renewcommand{\headrulewidth}{0pt}%
}
\thispagestyle{plain}

{\small
\tableofcontents
}

\section{Introduction}


Let $X$ be a complex manifold.
Recall that a holomorphic 2-form $\Omega\in H^0(X,\Omega^2_X)$
is symplectic if it is closed and non-degenerate
at every point of $X$. Viewing $\Omega$ as
a section of the bundle $\Lambda^{2,0}X$,
we note that one can uniquely recover the complex structure of $X$
from the form $\Omega$. Namely, the subbundle $T^{0,1}X$ is
the kernel of $\Omega$ and $T^{1,0}X$ is its complex conjugate.
We see that the complex 2-form $\Omega$ alone can be used
to encode both the complex structure and the symplectic structure
of $X$. Abstracting the properties of $\Omega$ that are necessary
to reconstruct the complex structure, we arrive at the 
notion of a C-symplectic form, see Definition \ref{def_C_sympl}.

A similar point of view on holomorphic symplectic structures
appeared in the work of Hitchin \cite{Hi1} and was further
explored in \cite{V2} and \cite{BDV}. One advantage of this
point of view is that we can describe
deformations of both the complex and symplectic structures on $X$
simultaneously.

Our aim is to investigate the properties
of C-symplectic structures further. Recall that the ordinary real symplectic
structures admit a simple local description. Namely, by
Darboux theorem, any real symplectic 
manifold $M$ is locally symplectomorphic to an open subset in a vector
space with the standard symplectic form. Moreover, any compact Lagrangian
submanifold $N\subset M$ admits a neighbourhood that is symplectomorphic
to a neighbourhood of the zero section in $T^*N$ with its canonical symplectic
form, see e.g. \cite[section 3]{MS}. To prove 
these statements one can use an idea that
goes back to Moser \cite{Mo} and construct an isotopy that maps
one symplectic form to another.

In the present paper, we prove some analogues of Moser's isotopy theorem
for C-symplectic structures (section \ref{sec_Moser})
and discuss their applications (section \ref{sec_appl}). It turns
out that it is useful to apply Moser's isotopy to holomorphic Lagrangian
fibrations. Given such a fibration $\pi\colon X\to S$ over a projective base $S$, we consider a degenerate
twistorial deformation of $X$ (introduced in \cite{V2}), 
which is a family of holomorphic symplectic
manifolds parametrized by $\bbC$ (see section \ref{sec_Lag} for details).
Such families have been classically studied in the case when $X$ is a K3
surface, see e.g. the discussion in \cite{Hu2}, where they are called
Brauer families. In \cite{Mar} these families
are studied in the case when $X$ is of $\mbox{K3}^{[n]}$-type. 

Our main observation is the following.
Given a hyperplane section $D\subset S$, we consider its open complement $U\subset S$.
If $\XX_t$ is the fibre of the degenerate twistorial deformation of $X$
over the point $t\in\bbC$,
it still admits a holomorphic Lagrangian fibration over $S$. Let $\XX_{U,t}$
be the preimage of $U$ in $\XX_t$. We prove in Theorem \ref{thm_main} that $\XX_{U,t}$
are isomorphic as complex manifolds for all $t\in \bbC$.
This has several implications.

First, we construct examples of complex
manifolds that admit infinitely many structures of a quasi-projective variety
that are pairwise non-birational, see Corollary \ref{cor_1}. Examples of complex
manifolds that admit several non-isomorphic algebraic structures are well known,
see e.g. \cite[Chapter VI, \S 3]{Ha} and \cite{Je}.
One may also compare our examples with those arising from non-abelian
Hodge theory of Hitchin--Simpson. Recall that de Rham and Betti moduli spaces
are isomorphic as complex analytic varieties (the Riemann--Hilbert
correspondence), but not as algebraic varieties, see e.g. \cite[Proposition 9]{Si}.

Next, we observe that any fibre of a holomorphic Lagrangian fibration
admits a Zariski-open neighbourhood that has a structure of a quasi-projective
variety, see Corollary \ref{cor_2}. In particular, any Lagrangian fibration is a Zariski-locally
projective morphism. This strengthens a result of Campana \cite{Ca} who
proved analogous statement in the analytic topology. Using the same idea,
we show in section \ref{sec_Stein} that any hyperk\"ahler manifold admitting
a Lagrangian fibration contains an open dense subset that is Stein. This
is related to a question of Bogomolov about the existence of ``Stein cells''
in compact complex manifolds, see \cite{B2}.
We use quasi-projective neighbourhoods to generalize to the non-projective setting the results
of Matsushita \cite{Ma3} about higher direct images of the structure sheaf,
see Corollary \ref{cor_3}. This shows how to obtain Matshushita's theorem in full
generality and in elementary way, without using more complicated techniques of Saito \cite{Sa}.

In section \ref{sec_contr} we consider a holomorphic symplectic manifold that
admits a bimeromorphic contraction such that the contraction centre is
mapped to a point. We use Moser's isotopies to show that such contraction
centres are rigid. More precisely, in a family of such contractions
the contraction centres of nearby fibres are isomorphic, and moreover the isomorphisms
are induced by symplectomorphisms of open neighbourhoods of the contraction centres,
see Theorem \ref{thm_contr}.

The Moser isotopy method which we describe
can be considered as an analytic version of the
work \cite{_KV:period_}, which deals with
formal deformations of holomorphic
symplectic manifolds, not necessarily compact. 
In  \cite{_KV:period_} it was shown that any 
deformation of a holomorphic symplectic
manifold with $H^{>0}(X, \calo_X)=0$
is trivial in the formal category, as
long as the de Rham cohomology class of its 
holomorphic symplectic form does not change.
This theorem, which has many useful applications, 
is in many aspects deficient. Indeed, a formal 
deformation of a Stein manifold or an affine manifold
is always trivial.
Using the Moser isotopy for
C-symplectic structures, we can
prove a local analytic version of this formal deformation
theorem, Corollary \ref{cor_Moser_compact}.



\section{C-symplectic structures}

\subsection{Definition and main properties}

We recall the definition of a C-symplectic structure from \cite{BDV}.

\begin{defn}\label{def_C_sympl}
Let $X$ be a real $4n$-dimensional $C^\infty$-manifold.
A C-symplectic form on $X$ is a smooth section
$\Omega$ of the complexified bundle of 2-forms $\Lambda^2_\bbC X$
that satisfies three conditions:
\begin{enumerate}
\item $d\Omega = 0$,
\item $\Omega^{n+1} = 0$,
\item $\Omega^n\wedge\overline{\Omega}\mathstrut^n \neq 0$ pointwise on $X$.
\end{enumerate}
\end{defn}

Decompose a C-symplectic form into its real and imaginary parts:
$\Omega = \omega_1 +\ii\omega_2$. Recall the following properties
(see \cite{BDV}).

\begin{enumerate}
\item Both $\omega_1$ and $\omega_2$ are real symplectic forms.
\item The kernel of $\Omega$ has rank $2n$ at all points of $X$.
Let the subbundle $T^{0,1}X \subset T_{\bbC}X$ be the
kernel of $\Omega$ and $T^{1,0}X$ its complex conjugate. Then
$T_{\bbC}X = T^{1,0}X\oplus T^{0,1}X$.
\item The subbundle $T^{1,0}X$ is closed under the Lie bracket.
This defines an integrable almost-complex structure on $X$. Therefore
any manifold with a C-symplectic form admits an intrinsically defined
complex structure.
\item The complex structure operator $I\in \emrp(TX)$ can be
expressed as follows: $I = \omega_2^{-1}\circ\,\omega_1$. Here we view
the 2-forms $\omega_j$ as morphisms $TX\to \Lambda^1 X$ defined by
contracting $\omega_j$ with the tangent vectors.
\item With respect to the above complex structure, the form $\Omega$
is holomorphically symplectic.
\end{enumerate}

\subsection{Lagrangian fibrations}\label{sec_Lag} We will be interested in certain families
of C-sym\-plectic structures that arise from holomorphic Lagrangian fibrations.
Let $X$ be a complex manifold of dimension $2n$ and $S$ a normal complex analytic
variety of dimension $n$.
Let $\pi\colon X\to S$ be a proper surjective holomorphic morphism with connected fibres.
Let $S^\circ \subset S$ be the maximal open subset over which the morphism $\pi$
is smooth. Recall that by Sard's theorem $S^\circ$ is non-empty and dense in $S$.
We denote by $\pi^\circ$ the restriction of $\pi$ to $X^\circ = \pi^{-1}(S^\circ)$.

Assume that $\Omega\in\Lambda^{2,0}X$ is a holomorphic symplectic form. 
We will call $\pi$ as above a {\it holomorphic Lagrangian fibration} if for any fibre $F$ of $\pi^\circ$
we have $\Omega|_F = 0$. Since the fibres are $n$-dimensional,
the latter condition means that $F$ is Lagrangian.
It is known that $F$ is a complex torus, see e.g. \cite[Theorem 2.6]{DM}.

\begin{rem} The base $S$ is smooth in all known examples of Lagrangian fibrations that are of
interest for us. However this condition is not
really necessary, so we do not include it in the definition.
We will need to use $C^\infty$ differential forms on $S$, and in the case of singular $S$
we define them as in \cite{Va}.
\end{rem}

Let $\eta\in \Lambda^{2,0}S\oplus \Lambda^{1,1}S$ be a closed 2-form.
For any $t\in \bbC$ define $\Omega_t = \Omega + t\pi^*\eta$.

\begin{thm}[\cite{BDV}]
In the above setting, the 2-forms $\Omega_t$ have the following properties.
\begin{enumerate}
\item $\Omega_t$ is a C-symplectic form on $X$ for any $t\in \bbC$.
\item If $I_t$ denotes the complex structure corresponding to $\Omega_t$,
then $\pi$ is a holomorphic Lagrangian fibration from $(X,I_t)$ to $S$ for any $t\in \bbC$.
\end{enumerate}
\end{thm}
\begin{proof}
{\it Part (1).} It is clear that $\Omega_t$ is closed. We will show that $\Omega_t^{n+1} = 0$
and $\Omega_t^n\wedge\overline{\Omega}\mathstrut_t^n = \Omega^n\wedge\overline{\Omega}\mathstrut^n$
for any $t\in \bbC$. It is enough to check both equalities on the dense open subset $X^\circ$.

The tangent space at any point $x\in X^\circ$ can be decomposed into a direct sum
of two Lagrangian subspaces: $T_x^{1,0}X \simeq T^{1,0}_x F \oplus W$, where $F$ is the
fibre of $\pi$ through $x$. We can choose bases in the two Lagrangian subspaces that are dual
with respect to $\Omega$: $T^{1,0}_x F = \langle e_1,\ldots,e_n\rangle$, $W = \langle f_1,\ldots,f_n\rangle$,
so that $\Omega(x) = \sum_j e_j^*\wdg f_j^*$.
Assume that $\alpha\in \Lambda^{p,q} S$. Then $(\pi^*\alpha)(x) \in \Lambda^pW^*\otimes \Lambda^q\overline{W}\mathstrut^*$
and from the formula for $\Omega$ it is clear that
\begin{equation}\label{eqn_zero}
\Omega^k\wdg \pi^*\alpha = 0 \quad\mbox{if}\quad p+k > n.
\end{equation}

The equality (\ref{eqn_zero}) clearly implies that $\Omega_t^{n+1} =0$.
We write $\eta = \eta_1 + \eta_2$ with $\eta_1\in\Lambda^{2,0}S$ and $\eta_2\in \Lambda^{1,1}S$ and consider the
form $\Omega_t^n\wedge\overline{\Omega}\mathstrut_t^n = (\Omega + t\pi^*\eta)^n\wedge(\overline{\Omega} + \bar{t}\pi^*\overline{\eta})^n$.
Opening the brackets, we obtain a sum, where the summands equal up to a constant
$\Omega^{k_1} \wedge \overline{\Omega}\mathstrut^{k_2} \wedge \pi^*\alpha$
with $\alpha = \eta_1^{l_1} \wedge \eta_2^{m_1} \wedge \bar{\eta}_1^{l_2} \wedge \bar{\eta}_2^{m_2}$
and $k_1+l_1+m_1 = k_2 + l_2 + m_2 = n$. The form $\alpha$
has Hodge type $(2l_1 + m_1 + m_2, 2l_2 + m_1 + m_2)$. Then the formula (\ref{eqn_zero})
and its complex conjugate imply that our summand is zero unless $l_1=m_1=l_2=m_2=0$.
This proves that $\Omega_t^n\wedge\overline{\Omega}\mathstrut_t^n = \Omega^n\wedge\overline{\Omega}\mathstrut^n\neq 0$
pointwise on $X$ and shows that $\Omega_t$ is a C-symplectic form.

{\it Part (2).} To prove that $\pi$ is a holomorphic map from $(X,I_t)$ to $S$ it is
again enough to restrict to the dense open subset $X^\circ$. Using the splitting
of the tangent space at arbitrary $x\in X^\circ$ as above,
we write $$(\pi^*\eta)(x) = \sum_{j,k}( a_{jk} f^*_j\wdg f^*_k + b_{jk} f^*_j\wdg \bar{f}^*_k).$$
Then it is clear that the kernel of $\Omega_t(x) = \sum_j e^*_j\wdg f^*_j + t(\pi^*\eta)(x)$ is spanned by the
vectors
$$
\bar{e}_j, \, j=1\ldots n \quad \mbox{and}\quad \bar{f}_j + t\sum_{l=1}^n b_{lj}e_l, \, j=1\ldots n.
$$
Since $e_l\in \ker{d\pi}$, we see that $d\pi (T^{0,1}_{I_t}X) \subset T^{0,1}S$, hence $d\pi$ is complex-linear and
$\pi$ is holomorphic. To prove that it is a Lagrangian fibration note that for any fibre $F\subset X^\circ$
we clearly have $\Omega_t|_F = 0$.
\end{proof}

In the case when
$X$ is a compact hyperk\"ahler manifold, it has been shown
in \cite{V2} that the complex manifolds $(X,I_t)$
form a smooth complex analytic family over 
the complex line $\bbC$ which we denote by $\XX$.
This family is called {\it degenerate twistorial deformation} of $X$. The terminology
can be explained as follows. Inside the period domain
$\DD$ the usual twistor conic can degenerate into the union
of a pair of projective lines. Removing the point of intersection of those
lines, we get a disjoint union of two affine lines. The periods of the fibres of the
degenerate twistorial deformation form one of such affine lines. 
The two lines are exchanged by complex conjugation,
i.e. one obtains the family parametrized by the second affine line
if one replaces the holomorphic form $\Omega$
with $\overline{\Omega}$ in the construction described above.

\subsection{Moser's isotopy}\label{sec_Moser} Recall the idea of Moser's isotopy theorem:
given a family of real symplectic forms $\omega_t$ parametrized by $t\in[0,1]$ 
such that the cohomology classes $[\omega_t]\in H^2(X,\bbR)$ do not depend on $t$, one constructs
a flow of diffeomorphisms $\varphi_t$ such that $\omega_0 = \varphi_t^*\omega_t$,
see \cite[Theorem 2]{Mo}.
Our aim is to prove a similar statement for C-symplectic forms.
However, no strict analogy is apparent, because certain additional
cohomological obstructions arise, and one needs to introduce some
further conditions to construct the isotopy.

We work in the following setting: $X$ is a manifold of real dimension $4n$,
$\Omega_t \in \Lambda^2_{\bbC}X$ is a family of C-symplectic forms
for $t\in [0,1]$. When talking about families of tensor fields, we always
assume that the dependence on $t$ is continuous and, if necessary, sufficiently differentiable.
We denote by $I_t$ the complex structure corresponding to $\Omega_t$
and by $\XX_t = (X,I_t)$ the corresponding complex manifold. If we use the $(p,q)$-decomposition
with respect to $I_t$, the corresponding bundles have subscript $I_t$, e.g. $\Lambda^{p,q}_{I_t}X$.
We denote by $\LL_{V}$ the Lie derivative along a vector field $V$.

\begin{lem}\label{lem_v}
In the above setting assume that there exists a family of complex 1-forms
$\alpha_t\in \Lambda^1_{\bbC} X$ such that for all $t\in [0,1]$ we have
\begin{enumerate}
\item $\alpha_t\in \Lambda^{1,0}_{I_t}X$,
\item $\frac{d}{dt}{\Omega}_t = d\alpha_t$.
\end{enumerate}
Then there exists a family of real vector fields $V_t\in TX$ such
that $\frac{d}{dt}{\Omega}_t = -\LL_{V_t}\Omega_t$.
\end{lem}
\begin{proof}
Since the forms $\Omega_t$ are closed, we have $\LL_{V_t}\Omega_t = d(V_t\cnv\Omega_t)$
for any real vector field $V_t$. Decompose $V_t = V_t^{1,0} + V_t^{0,1}$, where $V_t^{1,0} \in T^{1,0}_{I_t}X$
and $V_t^{0,1} = \overline{V_t}^{1,0}$, and note that 
$V_t\cnv\Omega_t = V^{1,0}_t\cnv\Omega_t\in \Lambda^{1,0}_{I_t}X$.
Since the 2-form $\Omega_t$ is symplectic for any $t$,
the map $W\mapsto W\cnv\Omega_t$ defines an isomorphism $T^{1,0}_{I_t}X \stackrel{\sim}{\to}\Lambda^{1,0}_{I_t}X$.
Applying the inverse of this isomorphism to $-\alpha_t$, we get the vector field $V_t^{1,0}$,
and its real part is the desired real vector field.
\end{proof}

In the setting of the above lemma, it remains to integrate the family
of vector fields $V_t$ in order to obtain isotopies that deform $\Omega_0$ into $\Omega_t$.

\begin{thm}[Moser's isotopy, version I]\label{thm_Moser_I}
In the above setting, assume additionally that
$H^1(\XX_t,\OO_{\XX_t}) = 0$ and $[\Omega_t] = [\Omega_0]\in H^2(X,\bbC)$ for all $t\in [0,1]$.
In this case:
\begin{enumerate}
\item There exists a family of vector fields $V_t$ such that $\frac{d}{dt}{\Omega}_t = -\LL_{V_t}\Omega_t$;
\item If $X$ is compact, then there exists a family of diffeomorphisms $\varphi_t$ such that
$\varphi_0 = \mathrm{id}$ and $\varphi_t^*\Omega_t = \Omega_0$ for $t\in [0,1]$.
\end{enumerate}
\end{thm}
\begin{proof}
{\it (1)} We need to find a family of 1-forms $\alpha_t$ that satisfies the
conditions of Lemma \ref{lem_v}. By our assumption $[\Omega_t] = [\Omega_0]$
for all $t\in [0,1]$, so $\Omega_{t+\varepsilon} - \Omega_t$ is exact. It follows that there
exists a family of 1-forms $\beta_t \in \Lambda^1_{\bbC}X$ such that $\frac{d}{dt}\Omega_t = d\beta_t$.
Differentiating the equation $\Omega_t^{n+1} = 0$ we see that $\Omega_t^n\wedge \frac{d}{dt}\Omega_t =0$.
Since $\Omega_t^n$ is a holomorphic volume form, we deduce that $\frac{d}{dt}\Omega_t \in \Lambda^{2,0}_{I_t}\oplus \Lambda^{1,1}_{I_t}$.
Decomposing $\beta_t = \beta_t^{1,0} + \beta_t^{0,1}$, we see that $\bar{\partial} \beta_t^{0,1} = 0$.
By our assumption $H^1(\XX_t,\OO_{\XX_t}) = 0$, therefore $\beta_t^{0,1} = \bar{\partial} f_t$ for
some $f_t$. Define $\alpha_t = \beta_t^{1,0} - \partial f_t$. Then
$d\alpha_t = d\beta_t^{1,0} - \ddb\dd f_t = d\beta_t^{1,0} + \dd \ddb f_t = d\beta_t$
and $\alpha_t$ is of type $(1,0)$ as required. By Lemma \ref{lem_v} we obtain
a family of vector fields $V_t$.

{\it (2)} 
Since $X$ is compact, we can integrate $V_t$ to a flow of diffeomorphisms $\varphi_t$, $t\in[0,1]$
with $\varphi_0 = \mathrm{id}$. We compute:
$$
\frac{d}{dt}(\varphi_t^*\Omega_t) = \varphi_t^*\left(\LL_{V_t}\Omega_t + \frac{d}{dt}\Omega_t\right) = 0
$$
by the definition of $V_t$ in part (1). Since $\varphi_0^*\Omega_0 = \Omega_0$,
we conclude that $\varphi_t^*\Omega_t = \Omega_0$ for all $t\in [0,1]$.
\end{proof}

This result can be applied to compact hyperk\"ahler manifolds,
which satisfy $H^1(X, \calo_X)=0$,
but in that case it easily follows from Bogomolov's version
of Kuranishi--Kodaira--Spencer theory, \cite{_Bogomolov:1978_}.
On non-compact manifolds, the assumption $H^1(X, \calo_X)=0$
is often quite restrictive. 
Moreover, on non-compact manifolds 
one needs some additional information to guarantee
that the flow $\varphi_t$ exists globally on $X$.
Without such information, integrating the vector field
$V_t$ in a neighbourhood of a compact subset,
one can deduce from Theorem \ref{thm_Moser_I} the following local statement.

\begin{cor}\label{cor_Moser_compact}
Let $\pi\colon \XX \to \Delta$ be
a smooth family of holomorphic
symplectic manifolds (not necessarily compact)
over the unit disc, trivial as a family of $C^\infty$ manifolds.
Denote by $\XX_t= \pi^{-1}(t)$ a fibre, and
by $\Omega_t\in H^0(\XX_t, \Omega^2_{\XX_t})$ its holomorphic
symplectic form. Using the $C^\infty$ trivialization to identify
cohomology groups of the fibres, assume that the cohomology
class of $\Omega_t$ does not depend on $t\in \Delta$,
and $H^1(\XX_t, \calo_{\XX_t})=0$. 
Let $K \subset \XX_{t_0}$ be a compact subset. Then there exists an open neighbourhood
$U\subset \Delta$  of $t_0\in \Delta$, and an open subset 
${\tilde U} \subset \pi^{-1}(U)$, with $K\subset {\tilde U}$,
with the following property. The set
${\tilde U}$ is locally trivially
fibred over $U$, with all fibres ${\tilde U} \cap \pi^{-1}(t)$,
$t\in U$ isomorphic as holomorphic symplectic
manifolds.
\end{cor}

It turns out that one can prove
a useful version of Moser's theorem in the situation when the family $\Omega_t$
comes from a Lagrangian fibration.

\begin{thm}[Moser's isotopy, version II]\label{thm_Moser_II}
Let $\pi\colon X\to S$ be a holomorphic Lagrangian fibration in the sense
of section \ref{sec_Lag}. Denote by $\Omega$ the holomorphic symplectic
form on $X$ and assume that $\alpha \in \Lambda^{1,0}S$. Let $\Omega_t = \Omega + t\pi^*(d\alpha)$
for $t\in[0,1]$ be the family of C-symplectic forms on $X$.
Then there exists a family of diffeomorphisms $\varphi_t$ of $X$ such that
$\varphi_0 = \mathrm{id}$ and $\varphi_t^*\Omega_t = \Omega_0$ for $t\in [0,1]$.
\footnote{A version of this theorem was independently obtained
by Abasheva and Rogov, to appear in \cite{_Abasheva:commu_}.}
\end{thm}
\begin{proof}
As we recalled in section \ref{sec_Lag}, the 2-forms $\Omega_t$
are C-symplectic and $\pi$ is a holomorphic Lagrangian
fibration on $(X,I_t)$. It follows that $\pi^*\alpha$
is in $\Lambda^{1,0}_{I_t}X$
for any $t\in[0,1]$. Hence we can apply Lemma \ref{lem_v}
with $\alpha_t = \pi^*\alpha$ and obtain a family of vector fields $V_t$.

We need to integrate $V_t$ to a flow of diffeomorphisms. Recall that by the construction
in the proof of Lemma \ref{lem_v} we have $V_t\cnv \Omega_t = -\alpha_t$.
Let $F$ be a smooth fibre of $\pi$ and $W \in T^{1,0}F$ some locally defined vector field.
Then $\Omega_t (V_t, W) = -\alpha_t(W) = -(\pi^*\alpha)(W) = 0$. Since $F$ is
Lagrangian and $W$ arbitrary, the last equality implies that $V_t$ is tangent to $F$.
So $d\pi(V_t)=0$ at all points of smooth fibres, hence everywhere on $X$, because smooth
fibres are dense.

We conclude that integral curves of $V_t$ are contained in the fibres of $\pi$.
Let now $F$ be an arbitrary fibre. By our assumptions about Lagrangian fibrations,
$F$ is compact. 
This implies that any integral curve of $V_t$ that starts at a point of $F$
can be extended for all $t\in[0,1]$. Since this is true for all fibres, we get
a well-defined flow of diffeomorphisms $\varphi_t$, $t\in[0,1]$ and
the rest of the proof goes as in Theorem \ref{thm_Moser_I}.
\end{proof}

\section{Applications}\label{sec_appl}

We apply Theorem \ref{thm_Moser_II} to Lagrangian fibrations
on compact hyperk\"ahler manifolds of maximal holonomy,
that is, compact simply connected K\"ahler manifolds $X$
with $H^0(X,\Omega^2_X)$ spanned by a symplectic form $\Omega$. 
Further on, we shall call these manifolds just 
``hyperk\"ahler manifolds''. We let $2n = \dim_{\bbC} X$.
Assume that $\pi\colon X\to S$ is a holomorphic Lagrangian
fibration in the sense explained in section \ref{sec_Lag}.
For the discussion of basic properties
of such fibrations we refer to \cite{Ma1}, \cite{Ma2} and \cite{Hu1}.
It is explained in \cite[footnote on page 53]{AC} that in this setting
the base $S$ is always projective.

\subsection{A degenerate twistorial deformation is Zariski locally trivial over the base}

We fix a projective
embedding $\iota\colon S\hrarr \bbP^m$, let $\omega_{FS} \in \Lambda^{1,1}\bbP^m$ be the Fubini--Study form
and $\eta = \iota^*\omega_{FS}$.
Setting $\Omega_t = \Omega + t\pi^*\eta$, $t\in \bbC$, we get a family of C-symplectic
structures that form a degenerate twistorial deformation of $X$. We use the notation
$\XX_t = (X,I_t)$ as above.

Let $H\subset \bbP^m$ be an arbitrary hyperplane and $U = S\setminus H$
its open complement in $S$. We denote by $\XX_{U,t}$ the preimage of $U$ in $\XX_t$
and $X_U = \XX_{U,0}$.

\begin{thm}\label{thm_main}
For any $t_1, t_2\in \bbC$ the complex manifolds $\XX_{U,t_1}$ and $\XX_{U,t_2}$ are isomorphic.
\end{thm}
\begin{proof}
It is enough to prove that $\XX_{U,t_1}$ is isomorphic to $\XX_{U,0}$. For this we shall
use Moser's isotopy theorem \ref{thm_Moser_II}.

We may choose local coordinates $(x_0:\ldots :x_m)$ on $\bbP^m$ such that $H$ is given
by $x_0 = 0$. In the open affine chart $x_0 \neq 0$ the Fubini--Study form is given by
$$
\omega_{FS}|_{\bbP^m\setminus H} = \ii\dd\ddb\log\left(1+\sum_{j=1}^m |x_j|^2\right)
$$
Note that $U$ is the intersection of $\bbP^m\setminus H$ with $S$.
We let $\alpha$ be the restriction of the 1-form $- \ii t_1 \dd \log(1+\sum_{j=1}^m |x_j|^2)$
to $U$. Then $\alpha\in \Lambda^{1,0}U$ and
$d\alpha = t_1\omega_{FS}|_U$. We can apply Theorem \ref{thm_Moser_II}, obtaining
a flow of diffeomorphisms $\varphi_t$, $t\in[0,1]$, such that $\varphi_1^*\Omega_{t_1}|_{\pi^{-1}(U)} = \Omega|_{\pi^{-1}(U)}$.
Then $\varphi_1$ defines an isomorphism of complex manifolds $\XX_{U,0}\stackrel{\sim}{\to} \XX_{U,t_1}$.
\end{proof}

\begin{rem}
The restriction of the Fubini-Study form to $S$ represents the cohomology
class of a hyperplane section, so this restriction is not an exact form.
Therefore to apply Moser's lemma it is necessary
to have $H\neq \emptyset$ in the above theorem. This condition is not
redundant, because the fibres $\XX_t$ are typically not pairwise isomorphic
to each other.
\end{rem}

Denote by $\mathcal{A}ut^s(X/S)$ the following sheaf of abelian groups in the Zariski topology
on $S$. For an open subset $U\subset S$ the sections of $\mathcal{A}ut^s(X/S)$ over $U$ are
complex analytic automorphisms of the complex manifold $X_U$ that commute with $\pi$
and preserve the symplectic form $\Omega|_{X_U}$.
Since a general fibre of $\pi$ is a complex torus, such automorphisms form an abelian group.

\begin{cor}\label{cor_twist}
For any $t\in \bbC$ the complex manifold $\XX_{t}$ can be obtained
as the twist of $X$ by a 1-cocycle with values in $\mathcal{A}ut^s(X/S)$.
\end{cor}
\begin{proof}
Keeping the notation as above, we cover $S$ by $m+1$ affine charts $U_i$
and denote by $\varphi_i\colon X_{U_i}\stackrel{\sim}{\to} \XX_{U_i,t}$ the
isomorphisms constructed in Theorem \ref{thm_main}. Then
$\psi_{ij} = \varphi_j^{-1}\circ\varphi_i \in \mathcal{A}ut^s(X/S)(U_i\cap U_j)$
define the 1-cocycle in question.
\end{proof}

Below we consider several applications of the above theorem.

\subsection{A complex manifold with several algebraic structures}
\label{_several_algebraic_Subsection_}
Recall that small deformations of compact K\"ahler manifolds remain K\"ahler
(\cite{_Kod-Spen-AnnMath-1960_}).
This implies that in the above setting for $t$ sufficiently close to zero the manifolds
$\XX_t$ are hyperk\"ahler. Using the projectivity criterion of Huybrechts
(\cite[Proposition 26.13]{Hu1}),
we can determine for which $t$ the manifolds $\XX_t$ are projective.
Assume that two projective fibres $\XX_{t_1}$ and $\XX_{t_2}$
are not birational.
Note that $\XX_{U,t_1}$ and $\XX_{U,t_2}$ are both quasi-projective.
Theorem \ref{thm_main} shows that $\XX_{U,t_1}$ and $\XX_{U,t_2}$ are isomorphic
as complex manifolds. But they can not be isomorphic (or even birational) as algebraic varieties,
otherwise $\XX_{t_1}$ and $\XX_{t_2}$ would be birational, contrary to our choice.

Starting from a K3 surface $X$ one can
produce a sufficiently general degenerate twistorial line in the moduli space of $X$,
and make sure that infinitely many fibres $\XX_{t_i}$ are projective and
pairwise non-birational. This gives us the following statement.

\begin{cor}\label{cor_1}
There exists a complex K3 surface $X$, a Zariski-open subset $\XX\subset X$ and infinitely many complex quasi-projective
varieties $\YY_j$ that are pairwise non-birational and such that $\YY_j^{\mathrm{an}}$
are isomorphic to $\XX$ as complex manifolds.
\end{cor}
\begin{proof}
For a complex K3 surface $X$, let $V_R = H^2(X,R)$
for $R = \bbZ, \bbQ, \bbR$ and $\bbC$. Let $q$ be the intersection form
on $V_\bbZ$. Let $\DD = \{x\in \bbP V_{\bbC}\st q(x) = 0, q(x,\bar{x})>0\}$
be the period domain. We assume that $\pi\colon X\to \bbP^1$ is a Lagrangian fibration
and denote by $\ell \in V_\bbZ$ the primitive cohomology class $[\pi^*\omega_{FS}]$.

{\it Step 1.} We first choose sufficiently general degenerate twistorial line.
Let $\DD_\ell = \DD \cap \bbP(\ell^\perp)$. The period domain has dimension $20$,
and $\DD_\ell$ is a hypersurface in it. The variety $\DD_\ell$ is an $\bbA^1$-bundle
over $\DD' = \{x\in \bbP(\ell^\perp/\langle\ell\rangle)\st q(x) = 0, q(x,\bar{x})>0\}$, the
fibres being the degenerate twistorial lines.

Let $W\subset V_\bbQ$ be a 3-dimensional subspace that contains $\ell$, and denote
by $\mathcal{S}$ the set of all such subspaces. This set is countable. Let $\DD_W = \DD \cap \bbP(W^\perp\otimes\bbC)$
and $\DD_W^+$ be the union of all degenerate twistorial lines in $\DD_\ell$ that pass through the
points of $\DD_W$. Since $q|_{W^\perp}$ is non-zero, the dimension of $\DD_W$ is $17$,
and the dimension of $\DD_W^+$ is not bigger than $18$. Hence $\DD_W^+$ is of positive codimension in $\DD_\ell$,
and $\DD_\ell^\circ = \DD_\ell \setminus \cup_{W\in \mathcal{S}} \DD_W^+$ is non-empty.
Deforming $X$ we may assume that its period $p\in \DD_\ell^\circ$.

{\it Step 2.} We consider the degenerate twistorial line through $[\Omega]$ and $\ell$.
Cohomology classes of C-symplectic forms on this line are of the form
$[\Omega_t] = [\Omega] + t\ell$. Recall that $V_\bbZ$ is the even unimodular 
lattice of signature $(3,19)$. Denote by $\HH(d)$ the rank two hyperbolic
plane with intersection form multiplied by $d$. 
By \cite[Appendix to \S 6, Theorem 1]{PS}, there exists a primitive
embedding of lattices $\HH(d)\hrarr V_\bbZ$, unique up to an automorphism of $V_\bbZ$. Conjugating such an embedding
by an automorphism of $V_\bbZ$, we find for arbitrary $d$ a primitive sublattice $\Lambda_d\subset V_\bbZ$
such that $\Lambda_d\simeq \HH(d)$ and $\ell\in \Lambda_d$.
Let $t_d$ be such that $[\Omega_{t_d}] \in \Lambda_d^{\perp}$.


{\it Step 3.} We consider the fibres $\XX_{t_d}$ of the degenerate twistorial family.
These fibres are complex K3 surfaces, in particular they are K\"ahler.
By our construction, $\mathrm{NS}(\XX_{t_d})$ contains $\Lambda_d$. The rank of $\mathrm{NS}(\XX_{t_d})$ cannot be bigger than 2 by our choice of $p$
made in Step 1. Hence $\mathrm{NS}(\XX_{t_d}) \simeq \Lambda_d$.
Since $\Lambda_d$ contains positive elements, $\XX_{t_d}$ are projective.
The discriminants of $\Lambda_d$ are different for different $d$, hence $\XX_{t_d}$ are pairwise
non-birational. Now it remains to set $\XX = \XX_{U,0}$, $\YY_j = \XX_{U,t_j}$,
where $U \simeq \bbA^1\subset \bbP^1$, and apply
Theorem \ref{thm_main}.
\end{proof}

\subsection{Local projectivity of Lagrangian fibrations}

We will sharpen the result of Campana \cite{Ca} about local projectivity of
Lagrangian fibrations. Let $\pi\colon X\to S$ be a Lagrangian fibration
with $X$ compact hyperk\"ahler and $S$ projective. It was shown in \cite{Ca} that any fibre $F$
of $\pi$ admits an open analytic neighbourhood $\VV$, such that $\pi|_\VV$
is a projective morphism.

\begin{cor}\label{cor_2}
In the above setting, $F$ admits a Zariski-open neighbourhood $\VV$
that is a quasi-projective variety and such that $\pi|_\VV$ is a projective
morphism.
\end{cor}
\begin{proof}
Since $S$ is projective, we can find a K\"ahler form $\eta\in \Lambda^{1,1}S$
whose cohomology class is rational. Let $\ell = [\pi^*\eta]\in H^2(X,\bbQ)$.
As before, we denote by $\Omega$ the holomorphic symplectic form on $X$.

Consider the degenerate twistorial deformation of $X$ determined by $\eta$.
Using the density of $H^2(X,\bbQ)$ in $H^2(X,\bbR)$, we can find a class $x\in H^2(X,\bbQ)$
with $q(x) > 0$ and $t = - q([\Omega],x)/q(\ell,x)$ arbitrary close to zero.
Then for the fibre $\XX_t$ of the degenerate twistorial family we have $[\Omega_t] = [\Omega] +t\ell$,
$x\in [\Omega_t]^\perp$, hence $\mathrm{NS}(\XX_t)$ contains an element with positive BBF square.
Since $|t|$ is small, $\XX_t$ is K\"ahler, and hence projective by Huybrechts's criterion \cite[Proposition 26.13]{Hu1}.
Pick a hyperplane section $H$ of $S$ that does not contain $\pi(F)$,
let $U = S\setminus H$ and $\VV = \pi^{-1}(U)$. Then apply Theorem \ref{thm_main}.
\end{proof}

\begin{rem}
This corollary gives, in particular, an alternative proof of the well known fact
that all fibres of $\pi$ are projective varieties.
\end{rem}

\begin{rem}
If $X$ is not projective, it can not be Moishezon, because being K\"ahler and
Moishezon implies being projective. Using the open subsets $\VV$ from Corollary \ref{cor_2},
we get an example of a compact complex manifold that admits a covering by
Zariski-open quasi-projective subsets, but is not Moishezon.
\end{rem}

\subsection{Stein cells in hyperk\"ahler manifolds with Lagrangian fibrations}\label{sec_Stein}

It was asked by Bogomolov (see \cite{B2}) whether any compact complex manifold $X$
admits a ``Stein cell''. By this one means an open subset $\VV\subset X$ that is
a Stein manifold, and such that the complement $X\setminus \VV$ is ``small''.
Ideally the complement of a Stein cell should be a finite CW complex whose dimension
is strictly smaller than the dimension of $X$. If $X$ is projective, one can
take for $\VV$ the complement of an ample divisor, which is an affine variety, hence Stein.
For non-algebraic varieties the question is more subtle, and no general construction
is known, even conjecturally. The condition on the complement of a Stein cell stated above
is quite restrictive, and it is debatable how realistic it is. We observe here that if one weakens
this condition, then Theorem \ref{thm_main} provides a way to construct Stein
cells in arbitrary (possibly non-algebraic) hyperk\"ahler manifolds admitting
a Lagrangian fibration.

\begin{cor}\label{cor_Stein}
Let $\pi\colon X\to S$ be a Lagrangian fibration with $X$ compact hyperk\"ahler
and $S$ projective. There exists an open subset $\VV\subset X$ that is Stein and
dense in the analytic topology.
\end{cor}
\begin{proof}
Consider the degenerate twistorial deformation of $X$ as in the proof of Corollary \ref{cor_2}
and let $\XX_t$ be a projective fibre. Let $H\subset S$ be a hyperplane section,
$U = S\setminus H$ and $\varphi\colon X_U\to \XX_{U,t}$ the isomorphism of complex manifolds
constructed in Theorem \ref{thm_main}. Let $Z = \pi^{-1}(H)\subset \XX_t$. For any ample
divisor $A\subset \XX_t$ the divisor $mA+Z$ is ample for $m$ big enough. Hence
$\VV' = \XX_{U,t}\setminus A$ is an affine variety. Let $\VV = \varphi^{-1}(\VV')$, then
$\VV$ is an open dense Stein subset of $X$.
\end{proof}

\begin{rem}
In the notation of the proof above, note that the closure of $\varphi^{-1}(A\cap \XX_{U,t})\subset X_U$
may not be a complex-analytic subset of $X$. So the complement of $\VV$ constructed above
may be more complicated than a CW complex.
\end{rem}

\subsection{Higher direct images of the structure sheaf}

Let $\pi\colon X\to S$ be a Lagrangian fibration
with $X$ compact hyperk\"ahler and $S$ a smooth projective variety. In the case when $X$ is projective,
Matsushita has shown in \cite{Ma3} that $R^j\pi_*\OO_X \simeq \Omega^j_S$. One can deduce the same result for non-projective $X$
using the theory of Saito \cite{Sa}. We show how to deduce the non-projective case
directly from \cite{Ma3} without using the complicated techniques from \cite{Sa}.

\begin{cor}\label{cor_3} In the above setting $R^j\pi_*\OO_X \simeq \Omega^j_S$ for all $j$.
\end{cor}
\begin{proof}
We let $\eta$ and $\ell$ be as in the proof of Corollary \ref{cor_2},
and we pick a projective fibre $\XX_t$ in the degenerate twistorial deformation
of $X$ as in that proof. We let $\omega\in \Lambda^{1,1}X$ be a K\"ahler form
and $L\in\mathrm{Pic}(\XX_t)$ be an ample line bundle.

Let $\pi^\circ\colon X^\circ\to S^\circ$ be the smooth part of the fibration $\pi$,
as in section \ref{sec_Lag}. We recall from \cite[section 2]{Ma3} the construction
of the isomorphism $R^1\pi^\circ_*\OO_{X^\circ} \simeq \Omega^1_{S^\circ}$.
Consider the isomorphism $T_{X^\circ}\stackrel{\sim}{\to} \Omega^1_{X^\circ}$
induced by the symplectic form $\Omega$. Compose this isomorphism with 
the natural surjection $\Omega^1_{X^\circ} \to \Omega^1_{X^\circ/S^\circ}$
and note that this composition vanishes on the subbundle $T_{X^\circ/S^\circ}$,
because the fibres of $\pi^\circ$ are Lagrangian. It therefore induces
an isomorphism $(\pi^\circ)^*T_{S^\circ}\stackrel{\sim}{\to} \Omega^1_{X^\circ/S^\circ}$.
Projecting it to $S$ we get an isomorphism $T_{S^\circ}\stackrel{\sim}{\to} \pi^\circ_*\Omega^1_{X^\circ/S^\circ}$.
We have an exact triple
$$
0\lrarr \pi^\circ_*\Omega^1_{X^\circ/S^\circ} \lrarr (R^1\pi^\circ_*\bbC)\otimes \OO_{S^\circ}\lrarr R^1\pi^\circ_*\OO_{X^\circ}\lrarr 0,
$$
and the bundle in the middle carries a polarized variation of Hodge structures, the
polarization being induced by the restriction of the K\"ahler class $[\omega]$ to the fibres of $\pi^\circ$.
Using the polarization, we get an isomorphism $R^1\pi^\circ_*\OO_{X^\circ} \simeq (\pi^\circ_*\Omega^1_{X^\circ/S^\circ})^\vee$
and hence
\begin{equation}\label{eqn_iso}
R^1\pi^\circ_*\OO_{X^\circ} \simeq \Omega^1_{S^\circ}.
\end{equation}

From the construction we see that the isomorphism (\ref{eqn_iso}) depends
on the class of the polarization $[\omega]\in H^2(X,\bbR)$. Recall, however,
that for any fibre $F$ of $\pi^\circ$ the restriction map $H^2(X,\bbR) \to H^2(F,\bbR)$
has rank one and its kernel is the orthogonal complement of $\ell$, see \cite[Lemma 2.2]{Ma4}.
The polarization on the fibres therefore depends only
on the value of $q(\ell,[\omega])$, and not on $[\omega]$ itself.
We can rescale $\omega$ so that $q(\ell,[\omega]) = q(\ell,c_1(L))$, where $L$ is
the ample bundle on $\XX_t$ chosen above. Then $[\omega]$ and $c_1(L)$ induce the same
polarization on the VHS $(R^1\pi^\circ_*\bbC)\otimes \OO_{S^\circ}$ and thus
the same isomorphism (\ref{eqn_iso}).

We need to show that (\ref{eqn_iso}) extends over the discriminant locus of $\pi$.
We choose an arbitrary hyperplane section of $S$ in some projective embedding
and denote by $U\subset S$ its open complement. We let $U^\circ = U\cap S^\circ$ and 
$X_U = \pi^{-1}(U)$. Theorem \ref{thm_main}
implies that the complex manifolds $X_U$ and $\XX_{U,t}$ are isomorphic.
Applying the results of \cite{Ma3} to $\XX_t$ and using the line bundle $L$
to define the polarization, we get an isomorphism
$R^1\pi_*\OO_{X_U} \simeq R^1\pi_*\OO_{\XX_{U,t}} \simeq \Omega^1_{U}$.
When restricted to $U^\circ$, this isomorphism coincides with the restriction of (\ref{eqn_iso}),
because of the remark about polarizations in the previous paragraph.
Therefore we can extend (\ref{eqn_iso}) to an isomorphism over $U$. Since $U$
was arbitrary, we conclude that (\ref{eqn_iso}) extends to an isomorphism 
$R^1\pi_*\OO_{X} \simeq \Omega^1_{S}$.

To finish the proof, we use the isomorphisms $\Lambda^j(R^1\pi_*\OO_X) \simeq R^j\pi_*\OO_X$
from \cite{Ma3}. These isomorphisms may be checked locally on $S$, so we again replace
$X_U$ by $\XX_{U,t}$ as above and apply the results of \cite{Ma3}.
\end{proof}

\subsection{Deformation invariance of contraction centres}\label{sec_contr}
Consider a smooth fa\-mily $\pi\colon \XX\to \Delta$
of holomorphic symplectic manifolds over the unit disc. Assume that $\ZZ\subset \XX$
is a subvariety that is flat over $\Delta$ and such that $\pi|_{\ZZ}$ is proper,
i.e. the fibres $\ZZ_t$ are compact subvarieties of $\XX_t$.
Assume that for every $t\in \Delta$ the subvariety $\ZZ_t$ can be contracted to a point.
More precisely, there exists a flat family $\pi'\colon \YY\to \Delta$
and a proper morphism $\rho\colon \XX\to \YY$ such that $\pi'\circ\rho = \pi$
and $\rho_t\colon \XX_t \to \YY_t$ is a bimeromorphic morphism with exceptional set $\ZZ_t$,
and such that $\rho_t(\ZZ_t)$ is a point in $\YY_t$. We call $\ZZ_t$ the contraction centre of $\rho_t$.

We would like to use Moser's isotopy to show that all $\ZZ_t$ are isomorphic,
the isomorphisms being induced by symplectomorphisms of some open neighbourhoods of $\ZZ_t$.
Every $\ZZ_t$ is a deformation retract of its open neighbourhood in $\XX_t$,
see e.g. \cite[Proposition 3.5]{Qu}. By shrinking $\XX$ we may assume that
every fibre $\ZZ_t$ is a deformation retract of $\XX_t$.
Moreover, since every point in $\YY_t$ admits a Stein neighbourhood, by shrinking $\YY$ we
may assume that all $\YY_t$ are Stein spaces. Therefore we are reduced to the
following statement.

\begin{thm}\label{thm_contr}
Let $X$ be a manifold with a family of C-symplectic structures $\Omega_t\in \Lambda^2_\bbC X$,
$t\in [0,1]$ and corresponding complex structures $I_t$. Denote $\XX_t = (X,I_t)$ and
assume that $\ZZ_t\subset \XX_t$ is a family of compact complex subvarieties such that:
\begin{enumerate}
\item $\ZZ_t$ is a deformation retract of $X$;
\item $\XX_t$ admits a bimeromorphic morphism $\rho_t$ onto a Stein space $\YY_t$ such that
the exceptional set of $\rho_t$ is $\ZZ_t$.
\end{enumerate}
Then there exist open neighbourhoods $U_0$ of $\ZZ_0$, $U_1$ of $\ZZ_1$ and a diffeomorphism
$\varphi\colon U_0 \to U_1$ such that $\varphi^*(\Omega_1|_{U_1}) = \Omega_0|_{U_0}$ and $\varphi(\ZZ_0) = \ZZ_1$.
\end{thm}
\begin{proof}
Since $\XX_t$ are holomorphic symplectic manifolds, $\YY_t$ are holomorphic symplectic
varieties in the sense of Beauville \cite{Be}. It is known \cite[Proposition 1.3]{Be} that the singularities
of $\YY_t$ are rational, and since $\YY_t$ are Stein, $H^j(\XX_t,\OO_{X_t}) = 0$ for all $j>0$.

As before, we denote by $[\Omega_t]$ the cohomology class of $\Omega_t$ in $H^2(X,\bbC)$.
We know from \cite[Corollary 2.8]{Ka} that $[\Omega_t]|_{\ZZ_t} = 0$. Since $\ZZ_t$ is a deformation
retract of $X$, we have $[\Omega_t] = 0$ for all $t$. Hence we can apply part (1) of Theorem \ref{thm_Moser_I}
and obtain a family of vector fields $V_t$.

Since $X$ is not compact, it is not possible, in general, to integrate $V_t$ to a flow of
diffeomorphisms. However, since $\ZZ_0$ is compact, we can integrate $V_t$ in an open
neighbourhood $U_0$ of $\ZZ_0$ for small $t$. More precisely, there exist diffeomorphisms
$\varphi_t$ of $U_0$ onto some open subsets of $X$ integrating $V_t$ and defined for
$t$ sufficiently close to zero. Note that $\varphi_t^*\Omega_t = \Omega_0|_{U_0}$ and
$\varphi_t(\ZZ_0)$ is a compact complex subvariety of $\XX_t$. Since $\XX_t$ is contractible
onto a Stein space $\YY_t$ with contraction centre $\ZZ_t$, the only positive dimensional
compact subvariety of $\XX_t$ is $\ZZ_t$. Hence $\varphi_t(\ZZ_0) = \ZZ_t$, and since all $\ZZ_t$ are compact,
we can extend the flow for all $t\in [0,1]$, after possibly shrinking $U_0$. Then
we define $U_1 = \varphi_1(U_0)$ and $\varphi = \varphi_1$.
\end{proof}

\medskip

{\bf Acknowledgements.} A.S. is grateful to Daniel Huybrechts for useful comments
on an earlier draft of this paper. Both authors are
grateful to Anna Abasheva and Vasily Rogov for
interesting discussions on the subject.

\end{document}